\documentclass[12pt, reqno]{amsart}
\usepackage{amssymb,amsthm,amsmath,cite}
\usepackage[colorlinks=true]{hyperref}
 \hypersetup{urlcolor=blue, citecolor=red}
 \usepackage{bookmark}

\newtheorem{lemma}{Lemma}[section]

\theoremstyle{definition}

\numberwithin{equation}{section}

\begin{document}
\title{On the higher moments of the error term in the Rankin-Selberg problem}
\author[J. Huang]{Jing Huang}
\address{School of Mathematics and Statistics, Shandong Normal
University\\
 Jinan 250014 \\Shandong\\ P. R. China\\}
\email{huangjingsdnu@163.com}

\author[Y. Tanigawa]{Yoshio Tanigawa}
\address{Nishisato 2-13-1\\Meito\\ Nagoya 465-0084\\ Japan}
\email{tanigawa@math.nagoya-u.ac.jp}

\author[W. Zhai]{Wenguang Zhai}
\address{Department of Mathematics,
China University of Mining and Technology\\ Beijing 100083\\ China\\}
\email{zhaiwg@hotmail.com}

\author[D. Zhang]{Deyu Zhang}
\address{School of Mathematics and Statistics, Shandong Normal
University\\
 Jinan 250014 \\Shandong\\ P. R. China\\}
\email{zdy\_78@hotmail.com}

\maketitle

{\bf Abstract}: Let $\Delta_1(x;\varphi)$ denote the error term in the classical Rankin-Selberg problem. In this paper, we consider the higher power moments
of $\Delta_1(x;\varphi)$ and derive the asymptotic formulas for 3-rd, 4-th and 5-th
power moments, which improve the previous results.



\renewcommand{\theequation}{\arabic{section}.\arabic{equation}}
\openup 1\jot

\section{Introduction}
\label{intro} \setcounter{equation}{0}
\medskip

 Let $\varphi(z)$ be a holomorphic cusp form of weight $\kappa$ with respect to the full modular group $SL(2,\mathbb{Z})$,
that is
\begin{equation}
\varphi(\frac{az+b}{cz+d})=(cz+d)^{\kappa}\varphi(z), \ \ (a,b,c,d\in\mathbb{Z}, ad-bc=1).
\end{equation}
We denote by  $a(n)$  the $n$-th Fourier coefficient of $\varphi(z)$ and suppose that $\varphi(z)$ is normalized such that $a(1) = 1$ and $T(n)\varphi = a(n)\varphi$ for
every $n \in \mathbb{N}$, where $T(n)$ is the Hecke operator of order $n$.
 Rankin \cite{R} and Selberg \cite{S} independently introduced the function
\begin{equation*}
Z(s)=\zeta(2s)\sum_{n=1}^\infty|a(n)|^2n^{1-\kappa-s},
\end{equation*}
where $\zeta(s)$ is the Riemann zeta-function. In the half plane $\sigma=\Re s>1$ the function $Z(s)$ has the absolutely convergent Dirichlet series expansion
\begin{equation*}
Z(s)=\sum_{n=1}^\infty c_nn^{-s},
\end{equation*}
where $c_n$ is the convolution function defined by
\begin{equation*}
c_n=n^{1-\kappa}\sum_{m^2|n}m^{2(\kappa-1)}\left|a\left(\frac{n}{m^2}\right)\right|^2.
\end{equation*}
In 1974, Deligne \cite{D} proved that $|a(n)|\leq n^{\frac{\kappa-1}{2}}d(n)$, where $d(n)$ is the Dirichlet divisor function. Then we know that $c_{n}\ll n^\varepsilon$.
Here and in what follows $\varepsilon$ denotes an arbitrarily small positive number which is not necessarily the same at each occurrence.
The classical Rankin-Selberg problem consists of the estimation of the error term
\begin{equation}\label{1.1}
\Delta(x;\varphi):=\sum_{n\leq x}c_n-Cx,
\end{equation}
where the constant $C$ may be written down explicitly (see e.g.\cite{IMT}).
In 1939, Rankin \cite{R} considered the analytic behaviour of $Z(s)$, and consequently he obtained
\begin{equation*}
\sum_{n\leq x}c_n=\frac{1}{6}\pi^2\kappa R_0x+\Delta(x;\varphi),
\end{equation*}
where
\begin{equation}\label{1.2}
\Delta(x;\varphi)=O(x^{\frac{3}{5}})
\end{equation}
and
\begin{equation*}
R_0=\frac{12(4\pi)^{\kappa-1}}{\Gamma(\kappa+1)}\int\int_ \mathfrak{F}y^{\kappa-2}|\varphi(z)|^2dxdy,
\end{equation*}
the integral being taken over a fundamental domain $\mathfrak{F}$ of $SL(2,\mathbb{Z})$. \eqref{1.2} was stated by Selberg \cite{S} again without proof. Recently, The exponent $\frac{3}{5}$ in \eqref{1.2} has been improved to $\frac{3}{5}-\delta, \delta>0$  by Huang \cite{Huang}, which is a breakthrough in the Rankin-Selberg problem. Since this kind sums play an very important role in the study of analytic number theory, many number theorists and scholars have obtained a series of meaningful research results (for example, see \cite{IS},\cite{TKM},\cite{TZZ},\cite{Zhai},\cite{ZZ}, etc.).

In 1999, Ivi\'{c}, Matsumoto and Tanigawa \cite{IMT} considered the Riesz mean of the type
\begin{equation*}
D_\rho(x;\varphi):=\frac{1}{\Gamma(\rho+1)}\sum_{n\leq x}(x-n)^\rho c_n
\end{equation*}
for any fixed $\rho\geq0$ and defined the error term $\Delta_\rho(x;\varphi)$ by
\begin{equation*}
D_\rho(x;\varphi)=\frac{\pi^2\kappa R_0}{6\Gamma(\rho+2)}x^{\rho+1}+\frac{Z(0)}{\Gamma(\rho+1)}x^\rho+\Delta_\rho(x;\varphi),
\end{equation*}
where
\begin{align*}
&R_0=\frac{12(4\pi)^{\kappa-1}}{\Gamma(\kappa+1)}\int\int_\mathfrak{F}y^{\kappa-2}|\varphi(z)|^2dxdy,
\end{align*}
the integral being taken over a fundamental domain $\mathfrak{F}$ of $SL(2,\mathbb{Z})$ and $Z(s)$ can be continued to the whole plane.

In \cite{IMT}, Ivi\'{c}, Matsumoto and Tanigawa
considered the relation between $\Delta(x;\varphi)$ and $\Delta_1(x;\varphi)$. For some $\alpha\geq0$, if $\Delta_1(x;\varphi)=O(x^\alpha)$, they obtained $\Delta(x;\varphi)=O(x^{\alpha/2})$. They also obtained
\begin{equation*}
\Delta_1(x;\varphi)=O(x^{\frac{6}{5}})
\end{equation*}
and
\begin{equation*}
\int_1^T\Delta_1^2(x;\varphi)dx=\frac{2}{13}(2\pi)^{-4}\left(\sum_{n=1}^\infty c_n^2n^{-\frac{7}{4}}\right)T^{\frac{13}{4}}+O(T^{3+\varepsilon}).
\end{equation*}

Tanigawa, Zhai and Zhang\cite{TZZ} studied the third, fourth and fifth power moments of $\Delta_1(x;\varphi)$ and proved that

\begin{equation}\label{1.3}
\begin{split}
\int_T^{2T}\Delta_1^3(x;\varphi)dx&=\frac{B_3(c)}{1120\pi^6}T^{\frac{35}{8}}+O\left(T^{\frac{35}{8}-\frac{1}{36}+\varepsilon}\right),\\
\int_T^{2T}\Delta_1^4(x;\varphi)dx&=\frac{B_4(c)}{11264\pi^8}T^{\frac{11}{2}}+O\left(T^{\frac{11}{2}-\frac{1}{221}+\varepsilon}\right),\\
\int_T^{2T}\Delta_1^5(x;\varphi)dx&=\frac{B_5(c)}{108544\pi^{10}}T^{\frac{53}{8}}+O\left(T^{\frac{53}{8}-\frac{1}{1731}+\varepsilon}\right),\\
\end{split}
\end{equation}
where
\begin{equation*}
\begin{split}
&B_k(f):=\sum_{l=1}^{k-1}\begin{pmatrix}  k-1 \\  l \end{pmatrix}s_{k;l}(f)\cos\frac{\pi(k-2l)}{4}, \\ &s_{k;l}:=\sum_{\sqrt[4]{n_1}+\cdots+\sqrt[4]{n_l}=\sqrt[4]{n_{l+1}}+\cdots+\sqrt[4]{n_k}}\frac{f(n_1)\cdots f(n_k)}{(n_1\cdots n_k)^{7/8}}, \ 1\leq l\leq k.
\end{split}
\end{equation*}

In this paper, we shall prove the following theorem, which improves \eqref{1.3}.

{\bf Theorem.}  {\it
Let $k\in \{3,4,5\}.$  We have the asymptotic formula
\begin{equation}\label{1.4}
\int_T^{2T}\Delta_1^k(x;\varphi)dx=\frac{B_k(c)}{(8+9k)2^{3k-4}\pi^{2k}} T^{1+9k/8}+O\left(T^{1+\frac{9k}{8}- \delta_k+\varepsilon}\right),
\end{equation}
where
$$\delta_3=\frac{3}{62},\ \delta_4=\frac{3}{256},\ \delta_5=\frac{1}{680}.$$
}

\section{Some Preliminary Lemmas}
\label{sec 2} \setcounter{equation}{0}
\medskip

\begin{lemma}\label{lemma2.1}
Let $x>1$ be a real number. For $1\ll N\ll x^2$ a parameter we have
\begin{equation}\label{2.1}
\Delta_1(x;\varphi)=\frac{1}{(2\pi)^2}\mathcal{R}_1(x;N)+O(x^{1+\varepsilon}+x^{3/2+\varepsilon}N^{-1/2}),
\end{equation}
where
\begin{equation*}
\mathcal{R}:=\mathcal{R}(x;N)=x^{9/8}\sum_{n\leq N}\frac{c_n}{n^{7/8}}\cos\left(8\pi\sqrt[4]{nx}-\frac{\pi}{4}\right).
\end{equation*}
\end{lemma}
\begin{proof}
This is {\cite[Lemma 2.1]{TZZ}}.
\end{proof}

\begin{lemma}\label{lemma2.2}
Let $k\geq3$, $(i_1,\cdots,i_{k-1})\in\{0,1\}^{k-1}$ such that
\begin{equation*}
\sqrt[4]{n_1}+(-1)^{i_1}\sqrt[4]{n_2}+(-1)^{i_2}\sqrt[4]{n_3}+\cdots+(-1)^{i_{k-1}}\sqrt[4]{n_k}\neq 0,
\end{equation*}
then we have
\begin{equation*}
|\sqrt[4]{n_1}+(-1)^{i_1}\sqrt[4]{n_2}+(-1)^{i_2}\sqrt[4]{n_3}+\cdots+(-1)^{i_{k-1}}\sqrt[4]{n_k}|\gg\max(n_1,\cdots,n_k)^{-(4^{k-2}-4^{-1})}.
\end{equation*}
\end{lemma}
\begin{proof}
This is {\cite[Lemma 2.3]{TZZ}}.
\end{proof}

\begin{lemma}\label{lemma2.3}
Suppose $T\geq 3$ is a large parameter,  $\alpha$ and $\beta$ are  real numbers such that
$\beta\not=0.$ Then we have
\begin{equation*}
\int_T^{2T}x^{\alpha}g(\beta x^{1/4})dx\ll  \frac{T^{\alpha+3/4}}{|\beta|},
\end{equation*}
where $g(x)=\cos(2\pi x), $ or $\sin(2\pi x),$ or $e(x):=e^{2\pi i x}.$
\end{lemma}
\begin{proof}
This follows from the first derivative test.
\end{proof}

\begin{lemma}\label{lemma2.4}
Suppose $k\geq3$, $(i_1,\cdots,i_{k-1})\in\{0,1\}^{k-1}$, $(i_1,\cdots,i_{k-1})\neq(0,\cdots,0)$, $N_1,\cdots,N_k>1$, $0<\Delta\ll H^{1/4}$, $H=\max(N_1,\cdots,N_k)$. Let $\mathcal{A}$ denote the number of solutions of the inequality
\begin{equation}\label{2.2}
|\sqrt[4]{n_1}+(-1)^{i_1}\sqrt[4]{n_2}+(-1)^{i_2}\sqrt[4]{n_3}+\cdots+(-1)^{i_{k-1}}\sqrt[4]{n_k}|<\Delta
\end{equation}
with $N_j<n_j\leq2N_j$, $1\leq j\leq k$,
where
\begin{equation*}
\mathcal{A}=\mathcal{A}(N_1,\cdots,N_k;i_1,\cdots,i_{k-1};\Delta).
\end{equation*}
Then we have
\begin{equation*}
\mathcal{A}\ll\Delta H^{-1/4}N_1\cdots N_k+H^{-1}N_1\cdots N_k.
\end{equation*}
\end{lemma}

\begin{proof}
The proof of this lemma is similar to the proof of {\cite[Lemma 2.4]{Z1}}. Suppose $H=N_k$. If $(n_1,\cdots,n_k)$ satisfies \eqref{2.2}, then for some $|\theta|<1$, we can obtain
\begin{equation*}
\sqrt[4]{n_1}+(-1)^{i_1}\sqrt[4]{n_2}+(-1)^{i_2}\sqrt[4]{n_3}+\cdots+(-1)^{i_{k-2}}\sqrt[4]{n_{k-1}}=(-1)^{i_{k-1}+1}\sqrt[4]{n_k}+\theta\Delta,
\end{equation*}
thus we have
\begin{equation*}
\left(\sqrt[4]{n_1}+(-1)^{i_1}\sqrt[4]{n_2}+(-1)^{i_2}\sqrt[4]{n_3}+\cdots+(-1)^{i_{k-2}}\sqrt[4]{n_{k-1}}\right)^4=n_k+O(\Delta N_k^{3/4}).
\end{equation*}
Therefore, for fixed $(n_1,\cdots,n_{k-1})$, the number of $n_k$ is $\ll1+\Delta N_k^{3/4}$ and so
\begin{equation*}
\mathcal{A}\ll\Delta N_k^{3/4}N_1\cdots N_{k-1}+N_1\cdots N_{k-1}.
\end{equation*}
\end{proof}

\begin{lemma}\label{lemma2.5}
 {\it  Let $c>0$ be a non-integer real number, $M\geq 2$ be a large parameter, $\delta>0$ be any real number.
Let $ \mathcal{A}(M,\delta;c)$ denote the number of solutions of the inequality
$$|m_1^c+m_2^c-m_3^c-m_4^c|\leq \delta M^c,\ \  M<m_1,m_2,m_3,m_4\leq 2M.$$
Then we have
$$  \mathcal{A}(M,\delta)\ll M^\varepsilon(M^{2} +\delta M^4  ).   $$
}
\end{lemma}

\begin{proof}   This is Theorem 2 of \cite{RS}.  \end{proof}

Let $T\geq 10$ be a large parameter and $y$ is a real number such that
$T^\varepsilon\ll y\ll T$. For any  $T\leq x\leq 2T$ define
\begin{eqnarray*}
&&\mathcal{R}_1(x;y):=\frac{x^{9/8}}{4\pi^2}\sum_{n\leq y}\frac{c_{n}}{n^{7/8}}\cos\left(8\pi(nx)^{1/4}-\frac{\pi}{4}\right),\\
 && \mathcal{R}_{2}(x;y):= \Delta_1(x;\varphi)- \mathcal{R}_{1}(x;y).\end{eqnarray*}

\begin{lemma}\label{lemma2.6}
  {\it
If $y\ll T^{1/2},$  then  we have the estimates
 \begin{eqnarray*}
 &&\int_T^{2T}| \mathcal{R}_{1}(x;y)|^{2\ell}dx\ll  T^{1+\frac{9\ell}{4}+\varepsilon},\ \ (\ell=1,2,3).
\end{eqnarray*}
}
\end{lemma}

\begin{proof} We only need to consider the case $\ell=3$. Using the large value technique  of \cite{TZZ} to $\mathcal{R}_1(x;y)$ directly we get
that  the estimate
$$
\int_T^{2T}| \mathcal{R}_{1}(x;y)|^{6}dx\ll T^{31/4+\varepsilon}
$$
holds for $y\ll T^{1/2}$.  We omit the details.   \end{proof}

\begin{lemma}\label{lemma2.7}
  {\it
If $y\ll T^{1/12},$  then  we have the estimates
 \begin{eqnarray}
 &&\int_T^{2T}| \mathcal{R}_{2}(x;y)|^{3}dx\ll  T^{\frac{35}{8}+\varepsilon}y^{-\frac 98},\\
&& \int_T^{2T}| \mathcal{R}_{2}(x;y)|^{4}dx\ll  T^{\frac{11}{2}+\varepsilon}y^{-\frac 32},\\
 && \int_T^{2T}| \mathcal{R}_{2}(x;y)|^{5}dx\ll  T^{\frac{53}{8}+\varepsilon}y^{-\frac 38}.
\end{eqnarray}

}
\end{lemma}

\begin{proof}
We estimate (2.4) first. For any   $T^{\varepsilon}\ll y\ll T^{1/3},$ we have the estimate
\begin{equation}
\int_T^{2T}| \mathcal{R}_{2}(x;y)|^2dx\ll \frac{T^{\frac{13}{4}+\varepsilon} }{y^{\frac 34}},
\end{equation}
which is  (4.8) in
  \cite{TZZ}.

By Lemma 2.1 we can write
$$ \mathcal{R}_{2}(x;y)=\frac{x^{9/8}}{  4\pi^2}\sum_{y<n\leq \sqrt T}\frac{c_{n}}{n^{7/8}}\cos(8\pi (nx)^{1/4}-\frac{\pi}{4})+\mathcal{R}_{2}(x;\sqrt{T}).$$
So we have
\begin{eqnarray}
\int_T^{2T}| \mathcal{R}_{2}(x;y)|^4dx\ll T^{\frac{9}{2}}\int_{T}^{2T}\left|\sum_{y<n\leq \sqrt T}\frac{c_{n}}{n^{7/8}}
e(4(nx)^{1/4} )\right|^4dx   + \int_T^{2T}| \mathcal{R}_{2}(x;\sqrt T)|^4dx.
\end{eqnarray}

By a splitting argument we get for some $y\ll N\ll \sqrt T$ that
\begin{eqnarray*}
&&\ \ \ \ \ T^{\frac{9}{2}}\int_{T}^{2T}\left|\sum_{y<n\leq \sqrt T}\frac{c_{n}}{n^{7/8}}
e(4(nx)^{1/4} )\right|^4dx  \\
&&\ll T^{\frac 92}  \log^4 T \times   \sum_{n_1,n_2,n_3,n_4\sim N}\frac{c_{n_1}c_{n_2}c_{n_3}c_{n_4}}{(n_1n_2n_3n_4)^{7/8}}\int_T^{2T}
e(4\rho\sqrt[4]{x})dx\nonumber\\
&&\ll T^{\frac{21}{4}}\log^4 T\times  \sum_{n_1,n_2,n_3,n_4\sim N}\frac{c_{n_1}c_{n_2}c_{n_3}c_{n_4}}{(n_1n_2n_3n_4)^{7/8}} \min\left(\sqrt[4]{T}, \frac{1}{|\rho|}\right),\nonumber
\end{eqnarray*}
where $\rho=n_1^{1/4} + n_2^{1/4}- n_3^{1/4} -n_4^{1/4}.   $

By Lemma 2.5, the contribution of $\sqrt[4]{T}$ (in this case $|\rho|\leq T^{-1/4}$) is
\begin{eqnarray*}
  &&\ll \frac{T^{11/2+\varepsilon}}{N^{7/2}} (N^2+T^{-1/4}N^{-1/4}N^4)\ll T^{11/2+\varepsilon}N^{-3/2}+T^{21/4+\varepsilon}N^{1/4}\\
  &&\ll T^{11/2+\varepsilon}y^{-3/2}+T^{21/4+\varepsilon}N^{1/4}
  \ll T^{11/2+\varepsilon}y^{-3/2}+T^{43/8+\varepsilon}\\&&\ll T^{11/2+\varepsilon}y^{-3/2}
  \end{eqnarray*}
if $y\ll T^{1/12}.$

Now we consider the contribution of $1/| \rho|$ for which $T^{-1/4}\ll |\rho| \ll y^{1/4}.$ We   divide the range of $\rho$ into
$O(\log T)$ subcases of the form $ \xi<|\rho|\leq 2 \xi$. By Lemma 2.5 again we get that the contribution of $1/|\rho|$ is
\begin{eqnarray*}
&&  \ll\frac{T^{21/4+\varepsilon}}{N^{7/2}} \max_{T^{-1/4}\ll \xi\ll y^{1/4}} \sum_{ \xi<|\rho|\leq 2\xi } \frac{1}{|\rho|} \\
&&  \ll   \frac{T^{21/4+\varepsilon}}{N^{7/2}} \max_{T^{-1/4}\ll \xi\ll y^{1/4}} \frac{1}{ \xi} (N^2+ \xi N^{-1/4}N^4) \\
&&\ll T^{11/2+\varepsilon}N^{-3/2}+T^{21/4+\varepsilon}N^{1/4}\ll T^{11/2+\varepsilon}y^{-3/2}.
\end{eqnarray*}

From the above three estimates we get
\begin{eqnarray}
&&  T^{\frac{9}{2}}\int_{T}^{2T}\left|\sum_{y<n\leq \sqrt T}\frac{c_{n}}{n^{7/8}}
e(4(nx)^{1/4} )\right|^4dx  \ll T^{11/2+\varepsilon}y^{-3/2}
\end{eqnarray}
for $y\ll T^{1/12}.$

By Lemma 2.1 we have   $ \mathcal{R}_2(x;\sqrt T)\ll T^{5/4+\varepsilon},$ which combining (2.6) gives
\begin{eqnarray}
\ \ \ \ \int_T^{2T}| \mathcal{R}_{2}(x;\sqrt T)|^4dx\ll T^{5/2+\varepsilon}\int_T^{2T}| \mathcal{R}_{2}(x;\sqrt T)|^2dx\ll
  T^{43/8+\varepsilon} \ll T^{11/2+\varepsilon}y^{-3/2}
\end{eqnarray}
by noting that $y\ll T^{1/12}.$

Now the estimate (2.4) follows from (2.7)-(2.9).  The estimate (2.3) follows from (2.6), (2.4) and Cauchy's inequality.

Now we estimate (2.5). From Lemma 2.6 with $\ell=3$ we get easily that
   the estimate
\begin{equation}
\int_T^{2T}| \mathcal{R}_{1}(x;y)|^{16/3}dx\ll T^{7+\varepsilon}
\end{equation}
holds for $y\ll T^{1/2}$. From \cite{TZZ} we have
\begin{equation}
\int_T^{2T}|  \Delta_{1}(x;\varphi)|^{16/3}dx\ll T^{7+\varepsilon}.
\end{equation}
From (2.10) and (2.11) we see that the estimate
\begin{equation}
\int_T^{2T}| \mathcal{R}_{2}(x;y)|^{16/3}dx\ll \int_T^{2T}| \Delta_{1}(x;\varphi)- \mathcal{R}_{1}(x;y)|^{16/3}dx\ll T^{7+\varepsilon}
\end{equation}
holds for $y\ll T^{1/2}$. Now by (2.4), (2.12) and H\"{o}lder's inequality we get that
$$
\int_T^{2T}|\mathcal{R}_{2}(x;y)|^{5}dx\ll \left(\int_T^{2T}| \mathcal{R}_{2}(x;y)|^{4}dx\right)^{1/4}
\left(\int_T^{2T}| \mathcal{R}_{2}(x;y)|^{16/3}dx\right)^{3/4}$$
$$\ll T^{53/8+\varepsilon}y^{-3/8}.
$$

\end{proof}

\section{Proof of Theorem}
\label{sec 3} \setcounter{equation}{0}
\medskip

In this section, we will give the proof of Theorem. Suppose $T\geq10$ is a real number.
It suffices for us to evaluate the integral $\int_T^{2T}\Delta_1^k(x;\varphi)dx$ for
 any $k\in \{3,4,5\}.$
 Suppose $y$ is a parameter such that $T^\varepsilon<y\leq T^{1/12}$.


\subsection{The evaluation of the integral} $\int_T^{2T} \mathcal{R}_1^k(x;y)dx$

Let ${I}=\{0,1\}, \ {\bf i}=(i_1,\cdots,i_{k-1})\in { I}^{k-1},\ {\bf n}=(n_1,\cdots, n_k)\in { N}^k.$
Define
\begin{align*}
\alpha(\boldsymbol{n};\boldsymbol{i}):& = \sqrt[4]{n_1}+(-1)^{i_1}\sqrt[4]{n_2}
+(-1)^{i_2}\sqrt[4]{n_3}+\cdots +(-1)^{i_{k-1}}\sqrt[4]{n_k},\\
\beta(\boldsymbol{i}): &=1+(-1)^{i_1}+(-1)^{i_2}+\cdots +(-1)^{i_{k-1}}.
\end{align*}

The formula (4.1) of \cite{TZZ} reads
\begin{equation}
\mathcal{R}_1^k(x;y)=\frac{1}{(2\pi)^{2k}2^{k-1}}\left(S_1(x)+S_2(x)\right),
\end{equation}
where
\begin{align*}
S_1(x):&= x^{9k/8}\sum_{\boldsymbol{i}\in {I}^{k-1}}\cos\left(-\frac{\pi\beta(\boldsymbol{i})}{4}
         \right) \sum_{\begin{subarray}{c} n_j\leq y,1\leq j\leq k \\
                       \alpha(\boldsymbol{n};\boldsymbol{i})=0\end{subarray}}
       \frac{ c_{n_1}\cdots c_{n_k}}{(n_1\cdots n_k)^{7/8}},\\
S_2(x):&= x^{9k/8}\sum_{\boldsymbol{i}\in  { I}^{k-1}}
        \sum_{\begin{subarray}{c} n_j\leq y,1\leq j\leq k \\
              \alpha(\boldsymbol{n};\boldsymbol{i})\not= 0 \end{subarray}}
       \frac{ c_{n_1}\cdots c_{n_k}}{(n_1\cdots n_k)^{7/8}} \cos\left(
        8\pi\alpha(\boldsymbol{n};\boldsymbol{i})\sqrt[4]x-\frac{\pi\beta(\boldsymbol{i})}{4}\right) .
\end{align*}

From (4.3) and (4.4) of \cite{TZZ} we have
\begin{equation}
\int_T^{2T}S_1(x)dx=B_k(c)\int_T^{2T}x^{\frac{9k}{8}}dx+O\left(T^{1+\frac{9k}{8}+\varepsilon}y^{-\frac 34}\right).
\end{equation}

We now estimate the contribution of $S_2(x).$ By Lemma 2.3 we have
\begin{equation}
\int_T^{2T}S_2(x)dx\ll T^{9k/8}\sum_{\boldsymbol{i}\in  { I}^{k-1}}
        \sum_{\begin{subarray}{c} n_j\leq y,1\leq j\leq k \\
              \alpha(\boldsymbol{n};\boldsymbol{i})\not= 0 \end{subarray}}
       \frac{ c_{n_1}\cdots c_{n_k}}{(n_1\cdots n_k)^{7/8}}\times \frac{T^{3/4}}{|\alpha(\boldsymbol{n};\boldsymbol{i})|}.
\end{equation}

The sum in the right-hand side of (3.3) can be divided into $O(\log^k T)$ sums of the form
\begin{eqnarray}
S(T;N_1,\cdots, N_k):&&=   T^{9k/8+3/4}\sum_{\boldsymbol{i}\in  {I}^{k-1}}
        \sum_{\begin{subarray}{c} n_j\sim N_j,1\leq j\leq k \\
              \alpha(\boldsymbol{n};\boldsymbol{i})\not= 0 \end{subarray}}
       \frac{ c_{n_1}\cdots c_{n_k}}{(n_1\cdots n_k)^{7/8}}\times \frac{1}{|\alpha(\boldsymbol{n};\boldsymbol{i})|} \nonumber\\
       && \ll \frac{T^{9k/8+3/4+\varepsilon}}{(N_1\cdots N_k)^{7/8}}\sum_{\boldsymbol{i}\in  {I}^{k-1}}
        \sum_{\begin{subarray}{c} n_j\sim N_j,1\leq j\leq k \\
              \alpha(\boldsymbol{n};\boldsymbol{i})\not= 0 \end{subarray}}
       \times \frac{1}{|\alpha(\boldsymbol{n};\boldsymbol{i})|}
       \end{eqnarray}
 by noting that $c_n\ll n^{\varepsilon} $,  where $1\ll N_j\ll y, \ 1\leq j\leq k.$ We only need to bound the sum
$$ W_{\bf i}(T;N_1,\cdots, N_k)= \sum_{\begin{subarray}{c} n_j\sim N_j,1\leq j\leq k \\
              \alpha(\boldsymbol{n};\boldsymbol{i})\not= 0 \end{subarray}}
         \frac{1}{|\alpha(\boldsymbol{n};\boldsymbol{i})|} $$
for each ${i}\in{I}^{k-1}.$

Let $H=\max(N_1,\cdots, N_k).$
If ${\bf i}=(0,\cdots,0),$ then
we have
\begin{eqnarray}
W_{\bf i}(T;N_1,\cdots, N_k)\ll \sum_{  n_j\sim N_j,1\leq j\leq k }
        \frac{1}{ n_1^{1/4}+\cdots + n_k^{1/4}} \ll \frac{N_1\cdots N_k}{H^{1/4}}.
\end{eqnarray}
Now we suppose ${\bf i}\not=(0,\cdots,0).$ The sum $W_{\bf i}(T;N_1,\cdots, N_k)$ can be divided into $O(\log T)$ sums for which
$0<\Delta<|\alpha(\boldsymbol{n};\boldsymbol{i})|\leq 2\Delta$. So by Lemma 2.4 we get
\begin{eqnarray}
W_{\bf i}(T;N_1,\cdots, N_k)&&\ll \frac{1}{\Delta}\left(\frac{\Delta N_1\cdots N_k}{H^{1/4}}+\frac{N_1\cdots N_k}{H}\right)\\
&&\ll \frac{  N_1\cdots N_k}{H^{1/4}}+\frac{N_1\cdots N_k}{\Delta H}\nonumber\\
&&\ll \frac{  N_1\cdots N_k}{H^{1/4}}+ N_1\cdots N_k  H^{4^{k-2}-  5/4}\nonumber\\
&&\ll N_1\cdots N_k  H^{4^{k-2}-  5/4}\nonumber,
\end{eqnarray}
where in the third step we used Lemma 2.2.

From (3.4)-(3.6)  we get that
\begin{eqnarray}
S(T;N_1,\cdots, N_k)\ll T^{3/4+9k/8+\varepsilon} H^{4^{k-2}-5/4+k/8}\ll  T^{3/4+9k/8+\varepsilon}y^{4^{k-2}-5/4+k/8}.
\end{eqnarray}

From (3.3) and (3.7) we get
\begin{equation}
\int_T^{2T}S_2(x)dx\ll T^{3/4+9k/8+\varepsilon} y^{4^{k-2}-5/4+k/8}.
\end{equation}

From (3.1), (3.2) and (3.8) we get
\begin{eqnarray}
\int_T^{2T}\mathcal{R}_1^k(x;y)dx&&=\frac{B_k(c)}{(2\pi)^{2k}2^{k-1}}\int_T^{2T}x^{\frac{9k}{8}}dx+O\left( T^{1+\frac{9k}{8}+\varepsilon}y^{-\frac 34} \right)\\
&&\ \ \ \ \ \ +O\left(T^{3/4+9k/8+\varepsilon} y^{4^{k-2}-5/4+k/8}\right).\nonumber
\end{eqnarray}

\subsection{\bf Estimate of}  $\int_T^{2T} \mathcal{R}_1^{k-1}(x;y)\mathcal{R}_2(x;y)dx$

We begin with the formula (4.13) of \cite{TZZ}, which reads
\begin{equation}
\int_T^{2T} \mathcal{R}_1^{k-1}(x;y) \mathcal{R}_2(x;y)dx=\int_T^{2T}
 \mathcal{R}_1^{k-1}(x;y) \mathcal{R}_2^{*}(x;y)dx+O(T^{1+\frac{9k}{8}-\frac
18+\varepsilon}),
\end{equation}
where
 $$
 \mathcal{R}_2^{*}(x;y)=(2\pi )^{-2}x^{\frac 98}\sum_{y<n\leq
T}\frac{c_n}{n^{7/8}} \cos(8\pi\sqrt[4]{nx}-\pi/4).
$$

Write
 \begin{eqnarray}
 \mathcal{R}_2^{*}(x;y)= \mathcal{R}_{21}^{*}(x;y)+\mathcal{R}_{22}^{*}(x;y),
\end{eqnarray}
where
\begin{eqnarray*}
&& \mathcal{R}_{21}^{*}(x;y)=(2\pi )^{-2}x^{\frac 98}\sum_{y<n\leq
2k y}\frac{c_n}{n^{7/8}} \cos(8\pi\sqrt[4]{nx}-\pi/4),\\
&& \mathcal{R}_{22}^{*}(x;y)=(2\pi )^{-2}x^{\frac 98}\sum_{2k y<n\leq
T}\frac{c_n}{n^{7/8}} \cos(8\pi\sqrt[4]{nx}-\pi/4).
\end{eqnarray*}

Similar to (4.1) of \cite{TZZ}, we can write
\begin{equation}
\mathcal{R}_1^{k-1}(x;y)\mathcal{R}_{21}^{*}(x;y)=\frac{1}{(2\pi)^{2k}2^{k-1}}\left(S_3(x)+S_4(x)\right),
\end{equation}
where
\begin{align*}
S_3(x):&= x^{9k/8}\sum_{\boldsymbol{i}\in {I}^{k-1}}\cos\left(-\frac{\pi\beta(\boldsymbol{i})}{4}
         \right) \sum_{\begin{subarray}{c} y<n_1<2ky,n_j\leq y,2\leq j\leq k \\
                       \alpha(\boldsymbol{n};\boldsymbol{i})=0\end{subarray}}
       \frac{ c_{n_1}\cdots c_{n_k}}{(n_1\cdots n_k)^{7/8}},\\
S_4(x):&= x^{9k/8}\sum_{\boldsymbol{i}\in  {I}^{k-1}}
        \sum_{\begin{subarray}{c} y<n_1<2ky,n_j\leq y,2\leq j\leq k \\
              \alpha(\boldsymbol{n};\boldsymbol{i})\not= 0 \end{subarray}}
       \frac{ c_{n_1}\cdots c_{n_k}}{(n_1\cdots n_k)^{7/8}} \cos\left(
        8\pi\alpha(\boldsymbol{n};\boldsymbol{i})\sqrt[4]x-\frac{\pi\beta(\boldsymbol{i})}{4}\right) .
\end{align*}

Similar to (4.14) of \cite{TZZ}, we have
\begin{align}
\int_T^{2T}S_3(x)dx\ll T^{1+\frac{9k}{8}+\varepsilon}y^{-\frac 34}.
\end{align}

Similar to (3.8), we have
\begin{align}
\int_T^{2T}S_4(x)dx\ll T^{3/4+9k/8+\varepsilon}  y^{4^{k-2}-5/4+k/8}.
\end{align}

Similar to (3.11), we have
\begin{equation}
\mathcal{R}_1^{k-1}(x;y)\mathcal{R}_{22}^{*}(x;y)=\frac{1}{(2\pi)^{2k}2^{k-1}} S_5(x) ,
\end{equation}
where
\begin{align*}
S_5(x):&= x^{9k/8}\sum_{\boldsymbol{i}\in  {I}^{k-1}}
        \sum_{\begin{subarray}{c} 2ky<n_1\leq T,n_j\leq y,2\leq j\leq k \\
              \alpha(\boldsymbol{n};\boldsymbol{i})\not= 0 \end{subarray}}
       \frac{ c_{n_1}\cdots c_{n_k}}{(n_1\cdots n_k)^{7/8}} \cos\left(
        8\pi\alpha(\boldsymbol{n};\boldsymbol{i})\sqrt[4]x-\frac{\pi\beta(\boldsymbol{i})}{4}\right).
\end{align*}
Note that in this case we always have $|\alpha(\boldsymbol{n};\boldsymbol{i})|\gg n_1^{1/4}.$ By Lemma 2.3  we have
\begin{eqnarray*}
\int_T^{2T}S_5(x)dx&&\ll T^{3/4+9k/8}\sum_{2ky<n_1\leq T}\sum_{n_2\leq y,\cdots, n_k\leq y}\frac{c_{n_1}c_{n_2}\cdots c_{n_k}}{n_1^{9/8}(n_2\cdots n_k)^{7/8}}\\
&&\ll T^{3/4+9k/8+\varepsilon}y^{(k-2)/8}.
\end{eqnarray*}
Thus we have
\begin{eqnarray}
\int_T^{2T}   \mathcal{R}_1^{k-1}(x;y)\mathcal{R}_{22}^{*}(x;y)dx  \ll T^{3/4+9k/8+\varepsilon}y^{(k-2)/8}.
\end{eqnarray}

Collecting  (3.10)-(3.16) we get
\begin{eqnarray}
\int_T^{2T} \mathcal{R}_1^{k-1}(x;y) \mathcal{R}_2(x;y)dx&&\ll  T^{1+\frac{9k}{8}-\frac
18+\varepsilon}+T^{1+\frac{9k}{8}+\varepsilon}y^{-\frac 34}+T^{\frac 34+\frac{9k}{8}+\varepsilon}  y^{4^{k-2}-\frac 54+\frac k8}\nonumber\\
&&\ll T^{1+\frac{9k}{8}+\varepsilon}y^{-\frac 34}+T^{\frac 34+\frac{9k}{8}+\varepsilon}  y^{4^{k-2}-\frac 54+\frac k8}.
\end{eqnarray}

\subsection{\bf Proof of Theorem for the case} $k=3$

We have
\begin{eqnarray}
\Delta_1^3(x;\varphi)&&=(\mathcal{R}_1(x;y)+ \mathcal{R}_2(x;y ))^3\\
&&=\mathcal{R}_1^3(x;y)+3\mathcal{R}_1^2(x;y)\mathcal{R}_2(x;y)+3\mathcal{R}_1(x;y)\mathcal{R}_2^2(x;y)+\mathcal{R}_2^3(x;y).\nonumber
\end{eqnarray}

Taking $k=3$ in (3.9) we have
\begin{eqnarray}
\int_T^{2T}\mathcal{R}_1^3(x;y)dx=\frac{B_3(c)}{2^8\pi^6}\int_T^{2T}x^{\frac{27}{8}}dx+O\left( T^{ \frac{35}{8}+\varepsilon}y^{-\frac 34} \right)
  +O\left(T^{33/8+\varepsilon} y^{25/8}\right).
\end{eqnarray}

Taking $k=3$ in (3.10) we have
\begin{eqnarray}
\int_T^{2T} \mathcal{R}_1^2(x;y) \mathcal{R}_2(x;y)dx \ll   T^{ \frac{35}{8}+\varepsilon}y^{-\frac 34}+T^{33/8+\varepsilon} y^{25/8}.
\end{eqnarray}

From Lemma 2.6 with $\ell=1$, (2.4) of Lemma 2.7 and Cauchy's inequality we get
\begin{eqnarray}
\int_T^{2T} \mathcal{R}_1 (x;y) \mathcal{R}_2^2(x;y)dx \ll   T^{ \frac{35}{8}+\varepsilon}y^{-\frac 34}.
\end{eqnarray}

From (3.18)-(3.21) and (2.3) of Lemma 2.7 we get that
\begin{eqnarray}
\int_T^{2T}\Delta_1^3(x;\varphi)dx&&=\frac{B_3(c)}{2^8\pi^6}\int_T^{2T}x^{\frac{27}{8}}dx+O\left( T^{ \frac{35}{8}+\varepsilon}y^{-\frac 34} \right)
  +O\left(T^{33/8+\varepsilon} y^{25/8}\right)\nonumber\\
  &&=\frac{B_3(c)}{2^8\pi^6}\int_T^{2T}x^{\frac{27}{8}}dx+O\left( T^{ \frac{35}{8}-\frac{3}{62}+\varepsilon}  \right)
\end{eqnarray}
by choosing $y=T^{2/31}.$ Now the case $k=3$ of Theorem follows from (3.22) by a splitting argument.

\subsection{\bf Proof of Theorem for the case} $k=4$

We have
\begin{eqnarray}
\Delta_1^4(x;\varphi) =\mathcal{R}_1^4(x;y)+4\mathcal{R}_1^3(x;y)\mathcal{R}_2(x;y)+O(\mathcal{R}_1^2(x;y)\mathcal{R}_2^2(x;y)
 +\mathcal{R}_2^4(x;y)).
\end{eqnarray}

Taking $k=4$ in (3.9) we have
\begin{eqnarray}
\int_T^{2T}\mathcal{R}_1^4(x;y)dx=\frac{B_4(c)}{2^{11}\pi^8}\int_T^{2T}x^{\frac{9}{2}}dx+O\left( T^{ \frac{11}{2}+\varepsilon}y^{-\frac 34} \right)
  +O\left(T^{21/4+\varepsilon} y^{61/4}\right).
\end{eqnarray}

Taking $k=4$ in (3.17) we have
\begin{eqnarray}
\int_T^{2T} \mathcal{R}_1^3(x;y) \mathcal{R}_2(x;y)dx \ll   T^{ \frac{11}{2}+\varepsilon}y^{-\frac 34}+T^{21/4+\varepsilon} y^{61/4}.
\end{eqnarray}

From   Lemma 2.6 with $\ell=2$, (2.4) of Lemma 2.7 and Cauchy's inequality we get
\begin{eqnarray}
\int_T^{2T} \mathcal{R}_1^2 (x;y) \mathcal{R}_2^2(x;y)dx \ll   T^{ \frac{11}{2}+\varepsilon}y^{-\frac 34}.
\end{eqnarray}

From (3.23)-(3.26) and (2.4) of Lemma 2.7 we get that
\begin{eqnarray}
\int_T^{2T}\Delta_1^4(x;\varphi)dx&&=\frac{B_4(c)}{2^{11}\pi^8}\int_T^{2T}x^{\frac{9}{2}}dx+O\left( T^{ \frac{11}{2}+\varepsilon}y^{-\frac 34} \right)
  +O\left(T^{21/4+\varepsilon} y^{61/4}\right)\nonumber\\
  &&=\frac{B_4(c)}{2^{11}\pi^8}\int_T^{2T}x^{\frac{9}{2}}dx+O\left( T^{ \frac{11}{2}-\frac{3}{256}+\varepsilon}  \right)
\end{eqnarray}
by choosing $y=T^{1/64}.$ Now the case $k=4$ of Theorem follows from (3.27) by a splitting argument.

\subsection{\bf Proof of Theorem for the case} $k=5$

We have
\begin{eqnarray}
\Delta_1^5(x;\varphi) =\mathcal{R}_1^5(x;y)+4\mathcal{R}_1^4(x;y)\mathcal{R}_2(x;y)+O(|\mathcal{R}_1^3(x;y)|\mathcal{R}_2^2(x;y)
 +|\mathcal{R}_2^5(x;y)|).
\end{eqnarray}

Taking $k=5$ in (3.9) we have
\begin{eqnarray}
\int_T^{2T}\mathcal{R}_1^5(x;y)dx=\frac{B_5(c)}{2^{14}\pi^{10}}\int_T^{2T}x^{\frac{45}{8}}dx+O\left( T^{ \frac{53}{8}+\varepsilon}y^{-\frac 34} \right)
  +O\left(T^{51/8+\varepsilon} y^{507/8}\right).
\end{eqnarray}

Taking $k=5$ in (3.17) we have
\begin{eqnarray}
\int_T^{2T} \mathcal{R}_1^4(x;y) \mathcal{R}_2(x;y)dx \ll   T^{ \frac{53}{8}+\varepsilon}y^{-\frac 34}+T^{51/8+\varepsilon} y^{507/8}.
\end{eqnarray}

From   Lemma 2.6 with $\ell=3$, (2.4) of Lemma 2.7 and Cauchy's inequality we get
\begin{eqnarray}
\int_T^{2T} |\mathcal{R}_1^3(x;y)| \mathcal{R}_2^2(x;y)dx \ll   T^{ \frac{53}{8}+\varepsilon}y^{-\frac 34}.
\end{eqnarray}

From (3.28)-(3.31) and (2.5) of Lemma 2.7 we get that
\begin{eqnarray}
\int_T^{2T}\Delta_1^5(x;\varphi)dx&&=\frac{B_5(c)}{2^{14}\pi^{10}}\int_T^{2T}x^{\frac{45}{8}}dx+O\left( T^{ \frac{53}{8}+\varepsilon}y^{-\frac 38} \right)
  +O\left(T^{51/8+\varepsilon} y^{507/8}\right)\nonumber\\
  &&=\frac{B_5(c)}{2^{14}\pi^{10}}\int_T^{2T}x^{\frac{45}{8}}dx+O\left( T^{ \frac{53}{8}-\frac{1}{680}+\varepsilon}  \right)
\end{eqnarray}
by choosing $y=T^{1/255}.$ Now the case $k=5$ of Theorem follows from (3.32) by a splitting argument.

\bigskip
\end{document}